\documentclass[12pt]{article}
\usepackage{amsmath,amsthm,amsfonts,amssymb,latexsym,enumerate,graphicx,mathrsfs,mathabx,cancel, mathpazo, hyperref}
\usepackage[utf8]{inputenc}
\usepackage[dvipsnames]{xcolor}
\usepackage[capitalize,nameinlink,noabbrev]{cleveref}
\usepackage{hyperref}
\hypersetup{colorlinks, citecolor=blue, linkcolor=red}

\usepackage{titlesec}

\titleformat*{\section}{\large\bfseries}

\titleformat*{\subsection}{\normalsize\bfseries}

\newtheorem{thm}{Theorem}[section]
\newtheorem*{thm*}{Theorem}
\newtheorem{cor}[thm]{Corollary}
\newtheorem*{cor*}{Corollary}
\newtheorem{prop}[thm]{Proposition} 
\newtheorem{lem}[thm]{Lemma}
\newtheorem*{claim*}{Claim}

\newtheorem*{question*}{Question}

\theoremstyle{definition}

\theoremstyle{remark}

\DeclareMathSymbol{\shortminus}{\mathbin}{AMSa}{"39}

\makeatletter
\@tfor\next:=abcdefghijklmnopqrstuvwxyzABCDEFGHIJKLMNOPQRSTUVWXYZ\do{%
	\def\command@factory#1{%
		\expandafter\def\csname cal#1\endcsname{\mathcal{#1}}
		\expandafter\def\csname frak#1\endcsname{\mathfrak{#1}}
		\expandafter\def\csname scr#1\endcsname{\mathscr{#1}}
		\expandafter\def\csname bb#1\endcsname{\mathbb{#1}}
		\expandafter\def\csname rm#1\endcsname{\mathrm{#1}}
		\expandafter\def\csname bf#1\endcsname{\mathbf{#1}}
	}
	\expandafter\command@factory\next
}
\makeatother

\newcommand{\aut}[1]{{\mathrm{Aut}\left(#1\right)}}

\title{Hyperbolic hyperbolic-by-cyclic groups are cubulable}
\author{François Dahmani \and Suraj Krishna M S \and Jean Pierre Mutanguha}
\date{\today}

\begin{document}

\maketitle

\begin{abstract}
We show that the mapping torus of a hyperbolic group by a hyperbolic automorphism is cubulable. Along the way, we (i) give an alternate proof of Hagen and Wise's theorem that hyperbolic free-by-cyclic groups are cubulable, and (ii) extend to the case with torsion Brinkmann's thesis that a torsion-free hyperbolic-by-cyclic group is hyperbolic if and only if it does not contain $\mathbb{Z}^2$-subgroups.
\end{abstract}

\section{Introduction}\label{sec;intro}

 In this note, we prove the following:

\begin{thm*}[\cref{cor;main_result}]
 Hyperbolic hyperbolic-by-cyclic groups  are  cubulable.
\end{thm*}

A \emph{hyperbolic-by-cyclic} group is a semidirect product  $G\rtimes \mathbb{Z}$ of a hyperbolic group $G$ with the integers $\mathbb Z$. A group is \emph{cubulable} if it admits an isometric action on a  
CAT(0) cube complex that is cubical, proper, and cocompact. The repetition in the statement is intended: we assume that both $G$ and $G\rtimes \mathbb{Z}$ are hyperbolic (equivalently, $G$ is hyperbolic and $G\rtimes \mathbb{Z}$ does not contain $\mathbb{Z}^2$, see \cref{cor;alg_hyp}).
This restricts what $G$ can be.  

Emblematic cases of our theorem are known by outstanding works.
First and foremost, if $G$ is a closed surface group, then any hyperbolic extension $G\rtimes \mathbb{Z}$ is a closed hyperbolic 3-manifold group \cite{Thurston3dmflds}.
Its cubulation is due to  
independent works of Bergeron and Wise \cite{BW} --- using Kahn and Markovic's surface subgroup theorem \cite{KM}, and Dufour \cite{Dufour} --- using the immersed quasiconvex surfaces of Cooper, Long, and Reid \cite{CLR}.
Second, when $G$ is free, Hagen and Wise cubulated the mapping torus  $G\rtimes \mathbb{Z}$ of a fully irreducible hyperbolic automorphism~\cite{HW16}. 

Hagen and Wise also treat extensions of free groups by arbitrary hyperbolic automorphisms in \cite{HW15}, a notoriously difficult analysis. We  do not rely on, nor follow, that work. 
Instead, our proof uses the emblematic cases above in a telescopic argument that encompasses the case when $G$ is a torsion-free hyperbolic group (see \cref{prop;tf_cubulable}).
It provides a hopefully appreciated alternative. 

We adopt a relative viewpoint and bootstrap the relative cubulation of certain free-product-by-cyclic groups by the first two named authors 
\cite{DK23}; 
this uses recent work of Groves and Manning on improper actions on CAT(0) cube complexes \cite{GM} along with the malnormal combination theorem of Hsu and Wise \cite{HW}.
The need for the theory of train tracks (of free groups or free product automorphisms, see \cite{BH, FM}) is limited to absolute train tracks for the  fully irreducible case; it is encapsulated in the relative cubulation of free-product-by-cyclic groups \cite{DK23}.

For a hyperbolic group $G$ possibly with torsion, if there exists a hyperbolic extension $G\rtimes \mathbb{Z}$, then $G$ is virtually torsion-free (and residually finite) by \cref{prop;virtFP}. 
In particular, $G\rtimes \mathbb{Z}$ is virtually cubulable hyperbolic, and hence cubulable \cite[Lem.~7.14]{Wise_book}. As a consequence, we have:

\begin{cor*}
If a hyperbolic-by-cyclic group $\Gamma$ is hyperbolic, then:
\begin{enumerate}
    \item $\Gamma$ is virtually (compact) special \cite{Agol};
    \item $\Gamma$ is $\mathbb Z$-linear and its quasiconvex subgroups are separable \cite{Haglund_Wise}; 
    \item $\Gamma$ virtually surjects onto $F_2$ \cite{AntolinMinasyan};
    \item $\Gamma$ is conjugacy separable \cite{MinasyanZalesskii}; and
    \item $\Gamma$ admits Anosov representations \cite{DFWZ}.
\end{enumerate}
\end{cor*}

We end this introduction with a question. \cref{prop;virtFP} states that a hyperbolic group is virtually a free product of free and surface groups whenever it admits a hyperbolic automorphism. 
However, the converse is false as can be seen from a hyperbolic triangle group or the free product of two finite groups --- these have finite outer automorphism groups.

\begin{question*}
Is there an algebraic characterisation of hyperbolic groups that admit hyperbolic automorphisms?
\end{question*} 

Note that Pettet characterised virtually free groups with finite outer automorphism groups \cite{pettet_finite_out}.

\subsection*{Acknowledgements}
We thank Mark Hagen for an insightful discussion, Daniel Groves for the reference to \cite[Lem.~7.14]{Wise_book}, and the referee for helpful suggestions.
We are grateful for the environment provided by the CRM during the 2023 semester program "Geometric Group Theory".

Fran\c{c}ois Dahmani was supported by ANR GoFR - ANR-22-CE40-0004, and IRL CRM-CNRS\,3457, while 
Suraj Krishna was supported by ISF grant number 1226/19.

\section{Free factor systems}

A \emph{free decomposition} of a group $G$ is an isomorphism $G \cong A_1*\cdots*A_k*F_r$, where $k\geq 0$, $r\geq 0$, each \emph{peripheral free factor} $A_i$ is not trivial, and $F_r$ is free with rank~$r$.
We call $\calA= (A_1, \dots, A_k)$ a \emph{free factor system} of $G$;
it is \emph{proper} unless $k \le 1$ and $r = 0$.
The integer $k+r$ is the \emph{Kurosh co-rank} 
of the free factor system~$\calA$.
A nontrivial group is \emph{freely indecomposable} if its free factor systems have Kurosh co-rank~1.

Assume $G$ is finitely generated for the rest of this section.
A \emph{Grushko decomposition} of $G$ is a free decomposition whose free factor system $\calA$ has  maximal Kurosh co-rank and peripheral free factors $A_i$ are not $\mathbb Z$; in that case, we call $\calA$ a \emph{Grushko free factor system} and its Kurosh co-rank is the \emph{Kurosh--Grushko rank} of $G$.

Recall the preorder of free factor systems of $G$: a free factor system $\calB=(B_1, \dots, B_\ell)$ is lower than $\calA$ if 
each $B_j$ is conjugate in $G$ to a subgroup of some~$A_i$.
In this case, a free decomposition with peripherals~$\calA$ refines to one with peripherals~$\calB$  
(as seen by the actions of $A_i$ on $T_\calB$, a Serre tree whose nontrivial vertex stabilisers are exactly the conjugates of all $B_j$), and 
the Kurosh co-rank of $\calB$ is at least that of $\calA$ (see \cite[Lem.~1.1]{DL} for a similar argument); if it is equal, then $\calA$ is also lower than~$\calB$. 

Let $\calB=(B_1, \dots, B_\ell)$ be a free factor system of $G$. A \emph{proper $(G, \calB)$-free factor} is a nontrivial point stabiliser of a nontrivial action of $G$ on a tree, for which  edge stabilisers are  trivial, and in which each $B_j$ is elliptic. In other words, it is a peripheral free factor $A_i$ in a free factor system $\calA$ that is higher than $\mathcal{B}$ in the preorder.

A minimal free factor system in this preorder is a Grushko free factor system; it is unique up to the preorder's equivalence relation.
So any automorphism preserves the Grushko free factor system $(A_1, \dots, A_k)$, i.e. it sends each $A_i$ to a conjugate of some $A_j$. A free factor system is \emph{periodic} with respect to $\phi \in \aut G$ if some (positive) power of $\phi$ preserves it.

\begin{lem}\label{lem;maximal_weak_inv}
 Suppose $G$ is a finitely generated group. 
 If  $\calB=(B_1, \dots, B_\ell)$ is a proper free factor system, then each $B_i$ has Kurosh--Grushko rank strictly lower than the Kurosh--Grushko rank of $G$.
 
 If $G$ has Kurosh--Grushko rank $\ge 2$, then any automorphism $\phi \colon G \to G$ has a free factor system that is maximal among $\phi$-periodic proper free factor systems.
\end{lem} 

\begin{proof}

Since $\calB$ is proper, $G \cong B_i*H$ for some nontrivial group $H$. 
By uniqueness of the Grushko decomposition, the Kurosh--Grushko rank of $G$ is the sum of those of $B_i$ and $H$.

For the second assertion, as the Kurosh--Grushko rank is at least 2, the Grushko free factor system is proper and $\phi$-periodic.
Restricting to $\phi$-periodic proper free factor systems, any one with the lowest Kurosh co-rank is maximal in the preorder.
\end{proof}

\section{Ingredients}\label{sec;ingredients}
Let $G$ be a torsion-free group.
For this section, we assume:
\begin{itemize}
    \item a free factor system $\calB = (B_1, \ldots, B_\ell)$ has Kurosh co-rank $\ge 3$;
    \item an automorphism $\psi \colon G \to G$ preserves $\calB$, denoted $\psi \in \aut{G,\calB}$;
    \item $\psi \in \aut{G,\calB}$ is \emph{relatively fully irreducible}, i.e.~any $\psi$-periodic (up to conjugacy) proper $(G,\calB)$-free factor must be conjugate to some $B_i$;
    \item $\psi \in \aut{G,\calB}$ is \emph{relatively atoroidal}, i.e.~any $\psi$-periodic conjugacy class of nontrivial elements in $G$ intersects some $B_i$.
\end{itemize}

Here is an equivalent definition of relatively fully irreducible:

\begin{lem}\label{lem;fullyirred}
An automorphim $\psi \in \aut{G, \calB}$ is relatively fully irreducible if and only if ~$\calB$ is a maximal $\psi$-periodic proper free factor system.
\end{lem}
\begin{proof}
If some $\psi$-periodic proper free factor system $(A_1, \ldots, A_k)$ is strictly higher than $\calB = (B_1, \ldots, B_\ell)$ in the preorder, then some $A_i$ is a $\psi$-periodic proper $(G, \calB)$-free factor that is not conjugate to any $B_j$.

Conversely, if some $\psi$-periodic proper $(G, \calB)$-free factor $A_1$ is not conjugate to any $B_i$, then the $\psi$-periodic free factor system $(A_1)$ can be extended to a $\psi$-periodic proper free factor system $(A_1, \ldots, A_k)$  that is strictly higher than $\calB$ by including some (conjugates of) $B_i$.
\end{proof}

For $h \in G$, ${\rm ad}_{h} \colon G \to G$ denotes the inner automorphism $g \mapsto hgh^{-1}$.
For a peripheral free factor $B_i$, let $k_i \ge 1$ be the smallest integer such that $\psi^{k_i}(B_i) = g_i^{-1}B_i{g_i}$ for some $g_i \in G$.
The \emph{peripheral suspension} $B_i \rtimes \mathbb Z$ is the suspension of $B_i$ by ${\rm ad}_{g_i}\circ \psi^{k_i}|_{B_i} \colon B_i \to B_i$;
this group naturally embeds in $G\rtimes_\psi \mathbb{Z}$ --- one can verify using normal forms that the natural homomorphism $B_i \rtimes \mathbb \langle s \rangle \to G \rtimes_\psi \mathbb \langle t \rangle$ that maps $s \mapsto g_i t^{k_i}$ is injective.

The first two named authors recently gave a \emph{relative cubulation} (introduced in \cite{EGRelCube}) of the mapping torus of a relatively fully irreducible relatively atoroidal automorphism.
Their proof is adapted from Hagen and Wise's cubulation of hyperbolic irreducible free-by-cyclic groups \cite{HW16}.

\begin{thm}[cf. {\cite[Thm. 1.1]{DK23}}]\label{thm;relative_cubulation}
    Under this section's assumptions, the mapping torus $ G \rtimes_\psi \bbZ$ acts cocompactly on a CAT(0) cube complex, where each cell stabiliser is either trivial or conjugate to a finite index subgroup of some peripheral suspension $B_i \rtimes \mathbb Z$. \qed
\end{thm}

The cited theorem has an additional assumption, absence of twinned subgroups: 
two subgroups $H_1 \neq H_2$ of $G$ are \emph{twinned} in $\calB$ if they are conjugates of some $B_j, B_k$ and ${\rm ad}_{g}\circ \psi^{n}(H_i) = H_i~(i=1,2)$ for some $n \ge 1$ and $g \in G$.
This assumption ensures the family of peripheral suspensions is malnormal (for relative hyperbolicity \cite[Thm.~0.1]{DL}), but
Guirardel remarked that it is redundant:

\begin{lem}[Guirardel]
    As $\calB$ has Kurosh co-rank $\ge 3$ and $\psi \in \aut{G,\calB}$ is relatively fully irreducible, there are no twinned subgroups in $\calB$.
\end{lem}
\noindent Our proof of the lemma uses objects (expanding train tracks, limit trees, geometric trees of surface type) that we do not define here for the sake of brevity; we refer the reader to the cited literature for each.

\begin{proof}    
The automorphism $\psi$ is represented by an expanding irreducible train track (see \cite[Sec.~1.3]{DL}).    Projectively iterating the train track produces the limit $(G,\calB)$-tree $T$ and a $\psi$-equivariant expanding homothety $h \colon T \to T$ (see \cite[p.~232]{BFH}). Note that nontrivial point stabilisers of $T$ are $\psi$-periodic (up to conjugacy) by the finiteness of $G$-orbits of branch points in $T$ \cite[Cor.~5.5]{Horbez_The_Boundary} and the $\psi$-equivariance of $h$.
    
    Let $H \le G$ be a nontrivial nonperipheral point stabiliser of $T$ --- nonperipheral means the subgroup is not conjugate to some~$B_i$.
    Then no proper $(G,\calB)$-free factor contains $H$ --- otherwise, the smallest such factor would be nonperipheral and $\psi$-periodic, yet $\psi \in \aut{G,\calB}$ is relatively fully irreducible.
    Thus $T$ is geometric of surface type \cite[Sec.~6.2, Lem.~6.8]{Horbez_The_Boundary} and the point stabiliser $H$ is cyclic \cite[Prop.~6.10]{Horbez_The_Boundary}.
    As $H$ was arbitrary, all nonperipheral point stabilisers of $T$ are cyclic;
    therefore, there are no twinned subgroups in $\calB$ because they would generate a noncyclic nonperipheral $T$-elliptic subgroup by the $\psi$-equivariance of $h$. 
\end{proof}

We will use the following theorem of Groves and Manning to upgrade relative cubulations in the next section.

\begin{thm}[cf. {\cite[Thm. D]{GM}}]\label{thm;groves_manning}
If a hyperbolic group $\Gamma$ acts cocompactly on a CAT(0) cube complex so that cell stabilisers are quasiconvex and cubulable, then $\Gamma$ is cubulable. \qed
\end{thm}

\noindent The cited theorem has ``virtually special'' in place of ``cubulable''.
Since virtually cubulable hyperbolic groups are cubulable \cite[Lem.~7.14]{Wise_book}, the properties ``virtually special'' and ``cubulable'' are equivalent for hyperbolic groups by Agol's theorem \cite{Agol}.
In particular, for hyperbolic groups, being cubulable is a commensurability invariant.

Finally, for sporadic cases when the Kurosh co-rank is 2, we will need a specialisation of Hsu and Wise's malnormal combination theorem:

\begin{thm}[cf. {\cite[Cor. C]{HW}}]\label{thm;hsu_wise}
Suppose $\Gamma = \Gamma_1 *_{\langle c \rangle} \Gamma_2$ or $\Gamma_1 *_{\langle c \rangle}$ is hyperbolic and $\langle c \rangle$ is an infinite cyclic malnormal subgroup of $\Gamma$. If each $\Gamma_i$ is cubulable, then $\Gamma$ is cubulable. \qed
\end{thm}

\noindent The two decompositions can be stated together as: ``$\Gamma$ splits over $\langle c \rangle$.''

\section{The bootstrap}

The following proposition is due to Sela (see \cref{prop;DSG_vcsg} for a proof).

\begin{prop}[cf. {\cite[Cor. 1.10]{SelaStructure}}]\label{prop;free_prod_surfaces}
Assume $G$ is a torsion-free hyperbolic group and some extension $G\rtimes_\phi  \mathbb{Z}$ does not contain a copy of $\mathbb{Z}^2$. 
If $G$ is freely indecomposable, then it is the fundamental group of a closed surface. \qed 
\end{prop}

We may now prove the central result of this note:

\begin{thm}\label{prop;tf_cubulable}
 Let $G$ be a torsion-free hyperbolic group.
 If $G\rtimes_\phi \mathbb{Z}$ is hyperbolic, then it is cubulable.
\end{thm}
\begin{proof}
    We proceed by induction on the Kurosh--Grushko rank.

    If the Kurosh--Grushko rank of $G$ is $1$, then $G$ is freely indecomposable.
    By \cref{prop;free_prod_surfaces}, $G$ is a closed surface group and, by the classification of its automorphisms, 
    $\phi$ is pseudo-Anosov \cite[Thm.~5.5]{Thurston3dmflds}. 
    Then $G\rtimes_\phi \mathbb{Z}$ is famously the fundamental group of a closed hyperbolic 3-manifold \cite[Thm.~5.6]{Thurston3dmflds}
     and cubulable, as already mentioned in \cref{sec;intro}. 
     Assume $n \ge 2$ and the theorem holds for torsion-free hyperbolic groups of Kurosh--Grushko  rank $< n$. 
    
Let the Kurosh--Grushko rank of $G$ be $n$.
\cref{lem;maximal_weak_inv} provides a maximal $\phi$-periodic proper free factor system $\calB=(B_1, \dots, B_\ell)$, and each $B_i$ has Kurosh--Grushko rank $< n$.
As each peripheral free factor $B_i$ is quasiconvex in the hyperbolic group $G$, a closest point projection $G \to B_i$ is Lipschitz and extends (cosetwise) to a \emph{peripheral retraction} $G \rtimes_\phi \mathbb Z \to B_i \rtimes \mathbb Z$ to the peripheral suspension.
Since $\phi$ is a quasi-isometry, the peripheral retractions are Lipschitz by the Morse lemma (in~$G$) --- a variation of this idea appears in \cite[Sec.~3]{MahanCTmaps}.
Thus the peripheral suspensions are quasiconvex and hyperbolic.
By the induction hypothesis, each $B_i \rtimes \mathbb Z$ is cubulable.

We distinguish two cases. 
The first case is when the Kurosh co-rank of $\calB$ is at least $3$.
Some positive power $\psi$ of~$\phi$ preserves $\calB$ and, by \Cref{lem;fullyirred}, $\psi \in \aut{G,\calB}$ is relatively fully irreducible.
    Since $G \rtimes_\psi \mathbb Z$ is hyperbolic, it has no $\mathbb Z^2$-subgroups and there are no $\psi$-periodic conjugacy classes of nontrivial elements in~$G$.
    In particular, $\psi \in \aut{G,\calB}$ is relatively atoroidal. 
    By \cref{thm;relative_cubulation}, $ G \rtimes_\psi \bbZ$ 
    acts cocompactly on a CAT(0) cube complex, where each cell stabiliser is either trivial or conjugate to a finite index subgroup 
    of some quasiconvex cubulable $B_i \rtimes \mathbb Z$.
    Groves and Manning's \cref{thm;groves_manning} thus implies $G\rtimes_\psi \mathbb{Z}$ is cubulable. 
    It naturally embeds in $G\rtimes_\phi \mathbb{Z}$ with finite index, so the latter is also cubulable by \cite[Lem.~7.14]{Wise_book}. 

    The last case is when the Kurosh co-rank of $\calB$ is $2$.  There are three possibilities: $G$ is $B_1*B_2$, $B_1*F_1$, or $F_2$. We rule out the third possibility as $F_2 \rtimes \mathbb Z$ is never hyperbolic --- it is a classical theorem of Nielsen that any automorphism of $F_2$ maps the commutator of a basis to a conjugate of itself or its inverse \cite{Nielsen}. 
    To conclude, we will prove that $\Gamma = G \rtimes_\phi \langle t \rangle$ (virtually) satisfies the hypotheses of Hsu and Wise's \cref{thm;hsu_wise}, and hence is cubulable.
    Note that $\langle t \rangle$ is a maximal cyclic subgroup of $\Gamma$, and hence malnormal. It remains to show that $\Gamma$ splits over $\langle t \rangle$ as needed.

    In the first possibility, up to taking the square of $\phi$, we may assume that $\phi$ preserves the conjugacy classes of both $B_1$ and $B_2$. After conjugation (which does not change the mapping torus), we may assume it fixes $B_1$ (setwise) and, being an automorphism, it sends $B_2$ to a conjugate by an element of $B_1$. After further conjugation, it fixes both $B_1$ and $B_2$. Then the mapping torus $\Gamma = (B_1 * B_2) \rtimes_\phi \langle t \rangle \cong (B_1\rtimes \langle t \rangle) *_{\langle t \rangle} (B_2 \rtimes \langle t \rangle) $. 
    
    In the second possibility, we write $G=B_1*\langle s\rangle$. Up to taking the square of $\phi$ and composing with a conjugation, we may assume that $\phi(B_1) = B_1$ and $\phi(s) = sb$ for some $b\in B_1$.  
    Consider  $G\rtimes_\phi \langle t \rangle$, where one has the relation $t s t^{-1} = sb $, or written differently $s^{-1}ts=bt$. 
    Then, rewriting the presentation, one has that 
    \[\Gamma = (B_1*\langle s \rangle )\rtimes_\phi \langle t \rangle \cong (B_1\rtimes \langle t \rangle) *_{\langle t \rangle^{s} =\langle bt \rangle,}\] where the last operation is an HNN extension with a stable letter $s$ that (right) conjugates $\langle t \rangle $ to $\langle bt \rangle$ (and actually $t$ to $bt$).
    \end{proof}

\section{Once more, with torsion}
Now  $G$ is a finitely presented group (possibly with torsion). It has a maximal decomposition as the fundamental group of a finite graph of groups with finite edge groups \cite{Dunwoody}. The infinite vertex groups are thus one-ended \cite{Stallings}. We call this a \emph{Dunwoody--Stallings decomposition}. 
It is not unique, but the conjugacy classes of infinite vertex groups are uniquely defined: they are conjugacy classes of the maximal one-ended subgroups of $G$. 
 The following is a generalisation of \cref{prop;free_prod_surfaces}:

\begin{prop}\label{prop;DSG_vcsg}
  Assume $G$ is a hyperbolic group (possibly with torsion) and some extension $G \rtimes_\phi \mathbb Z$ does not contain a copy of $\mathbb Z^2$.
  Then every maximal one-ended subgroup of $G$ is virtually a closed surface group.
\end{prop}
\begin{proof}
Let $H$ be a  maximal one-ended subgroup of $G$. Since there are only finitely many conjugacy classes of such subgroups, $\psi = ({\rm ad}_g \circ \phi^k)|_{H}$ is an automorphism of~$H$ for some integer $k\geq 1$ and element $g\in G$.
 
 Similar to the discussion in \cref{sec;ingredients}, the suspension  $H\rtimes_{\psi} \mathbb{Z}$ naturally embeds in $G \rtimes_\phi\mathbb{Z}$. As $H$ is one-ended, its JSJ decomposition is preserved by $\psi$ \cite[Thm.~0.1]{Bowditch}. The lack of $\mathbb{Z}^2$ in $G \rtimes_\phi\mathbb{Z}$ imposes that the JSJ is trivial but not a rigid vertex \cite[Cor.~1.3]{BestvinaFeighn}. It is therefore a vertex of surface type. In particular, $H$ is virtually a closed surface group (see, for instance, \cite[Sec.~4]{Martino}).
\end{proof}

We are now ready to state the main observation of this section.
\begin{prop}\label{prop;virtFP}
  If $G$ is a hyperbolic group (possibly with torsion) and 
  some extension $G \rtimes_\phi \mathbb{Z}$ does not contain a copy of $\mathbb{Z}^2$, then $G$
  has a characteristic finite index subgroup that is a free product of closed surface groups and free groups. In particular, $G$ is residually finite.
  \end{prop}

\begin{proof} 
Let $\bbX$ be a Dunwoody--Stallings decomposition of $G$. We need notations for the decomposition: the underlying finite graph is $X$; for each vertex $v$ in $X$, its vertex group is $\mathbb{X}_v$; and for each edge $e$ in $X$, its finite edge group is $\mathbb{X}_e$. 
For each vertex $v$, denote by $H_v$ a normal finite index subgroup of $\bbX_v$ that is either trivial or a closed surface group, as guaranteed by \cref{prop;DSG_vcsg}.

As the subgroups $H_v$ are torsion-free, the surjections $q_v \colon \bbX_v \to \bbX_v/H_v$ are injective on finite subgroups.
Thus we define a graph of finite groups $\bbY$ with underlying graph $X$, vertex groups $\bbX_v/H_v$, and edge groups $\bbX_e$; the surjections $q_v$ induce a surjection $q \colon G \to \pi_1(\bbY)$ with a torsion-free kernel. The quotient $\pi_1(\bbY)$ is virtually free by Karrass, Pietrowski, and Solitar's characterisation \cite[Thm.~1]{KPS}. 

Let $J \le \pi_1(\bbY)$ be a free finite index subgroup.
Since $J$ and the kernel of~$q$ are torsion-free, the preimage $q^{-1}(J) \le G$ is a torsion-free finite index subgroup.
The intersection $H$ of subgroups of $G$ with index $[G : q^{-1}(J)]$ is a characteristic torsion-free finite index subgroup.
The decomposition~$\bbX$ of $G$ induces a Grushko decomposition of $H$ whose freely indecomposable free factors are closed surface groups.
\end{proof}

We may extend Brinkmann's thesis \cite{Brinkmann} to the case with torsion.
\begin{cor}\label{cor;alg_hyp}
  Suppose $G$ is a hyperbolic group. Then $G\rtimes_\phi \mathbb{Z}$ is hyperbolic if and only if it does not contain a copy of $\mathbb{Z}^2$. 
\end{cor}
The forward implication is standard. Conversely, if $G \rtimes_\phi \mathbb Z$ does not contain a copy of $\mathbb{Z}^2$, then the same holds for the finite index subgroup $G_0\rtimes_{\phi|_{G_0}} \mathbb{Z}$, where $G_0$ is the torsion-free subgroup given by \cref{prop;virtFP}. As  $G_0 \rtimes_{\phi|_{G_0}} \mathbb{Z}$ is hyperbolic \cite{Brinkmann}, so is $G\rtimes_\phi \mathbb{Z}$.

\begin{cor}\label{cor;main_result}
  If $G$ and $G \rtimes_\phi \mathbb{Z}$ are hyperbolic groups, then $G \rtimes_\phi \mathbb{Z}$ is cubulable.
\end{cor}

Again, consider the finite index subgroup $G_0\rtimes_{\phi|_{G_0}} \mathbb{Z}$ of $G\rtimes_\phi \mathbb{Z}$, where $G_0$ is given by 
\cref{prop;virtFP}. $G_0\rtimes_{\phi|_{G_0}} \mathbb{Z}$ is cubulable by \cref{prop;tf_cubulable}, and hence, by \cite[Lem.~7.14]{Wise_book}, so is $G\rtimes_\phi \mathbb{Z}$.

{\small 
\bibliographystyle{alpha}
\bibliography{refs}

\begin{thebibliography}{DFWZ23}

\bibitem[Ago13]{Agol}
Ian Agol.
\newblock The virtual {Haken} conjecture (with an appendix by {Ian} {Agol},
  {Daniel} {Groves} and {Jason} {Manning}).
\newblock {\em Doc. Math.}, 18:1045--1087, 2013.

\bibitem[AM15]{AntolinMinasyan}
Yago Antol{\'{\i}}n and Ashot Minasyan.
\newblock Tits alternatives for graph products.
\newblock {\em J. Reine Angew. Math.}, 704:55--83, 2015.

\bibitem[BF95]{BestvinaFeighn}
Mladen Bestvina and Mark Feighn.
\newblock Stable actions of groups on real trees.
\newblock {\em Invent. Math.}, 121(2):287--321, 1995.

\bibitem[BFH97]{BFH}
Mladen {Bestvina}, Mark {Feighn}, and Michael {Handel}.
\newblock Laminations, trees, and irreducible automorphisms of free groups.
\newblock {\em Geom. Funct. Anal.}, 7(2):215--244, 1997.

\bibitem[BH92]{BH}
Mladen Bestvina and Michael Handel.
\newblock Train tracks and automorphisms of free groups.
\newblock {\em Ann. Math. (2)}, 135(1):1--51, 1992.

\bibitem[Bow98]{Bowditch}
Brian~H. Bowditch.
\newblock Cut points and canonical splittings of hyperbolic groups.
\newblock {\em Acta Math.}, 180(2):145--186, 1998.

\bibitem[Bri00]{Brinkmann}
Peter Brinkmann.
\newblock {\em Mapping Tori of Automorphisms of Hyperbolic Groups}.
\newblock PhD thesis, University of Illinois Urbana-Champaign, 2000.

\bibitem[BW12]{BW}
Nicolas Bergeron and Daniel~T. Wise.
\newblock A boundary criterion for cubulation.
\newblock {\em Am. J. Math.}, 134(3):843--859, 2012.

\bibitem[CLR94]{CLR}
Daryl Cooper, Darren~D. Long, and Alan~W. Reid.
\newblock Bundles and finite foliations.
\newblock {\em Invent. Math.}, 118(2):255--283, 1994.

\bibitem[DFWZ23]{DFWZ}
Sami {Douba}, Balthazar {Fl{\'e}chelles}, Theodore {Weisman}, and Feng {Zhu}.
\newblock {Cubulated hyperbolic groups admit Anosov representations}.
\newblock {\em arXiv:2309.03695}, 2023.

\bibitem[DL22]{DL}
Fran{\c{c}}ois Dahmani and Ruoyu Li.
\newblock Relative hyperbolicity for automorphisms of free products and free
  groups.
\newblock {\em J. Topol. Anal.}, 14(1):55--92, 2022.

\bibitem[DM]{DK23}
François {D}ahmani and Suraj~Krishna {M S}.
\newblock Cubulating a free-product-by-cyclic group.
\newblock {\em arXiv:2212.09869. To appear in Alg. Geom. Topol.}

\bibitem[Duf12]{Dufour}
Guillaume Dufour.
\newblock {\em Cubulations de variétés hyperboliques compactes.
  Mathématiques générales}.
\newblock PhD thesis, Université Paris Sud - Paris XI, 2012.

\bibitem[Dun85]{Dunwoody}
Martin~J. Dunwoody.
\newblock The accessibility of finitely presented groups.
\newblock {\em Invent. Math.}, 81(3):449--457, 1985.

\bibitem[EG20]{EGRelCube}
Eduard Einstein and Daniel Groves.
\newblock Relative cubulations and groups with a 2-sphere boundary.
\newblock {\em Compos. Math.}, 156(4):862--867, 2020.

\bibitem[FM15]{FM}
Stefano Francaviglia and Armando Martino.
\newblock Stretching factors, metrics and train tracks for free products.
\newblock {\em Ill. J. Math.}, 59(4):859--899, 2015.

\bibitem[GM]{GM}
Daniel {Groves} and Jason~F. {Manning}.
\newblock {Hyperbolic groups acting improperly}.
\newblock {\em arXiv:1808.02325. To appear in Geom. Topol.}

\bibitem[Hor17]{Horbez_The_Boundary}
Camille Horbez.
\newblock The boundary of the outer space of a free product.
\newblock {\em Isr. J. Math.}, 221(1):179--234, 2017.

\bibitem[HW08]{Haglund_Wise}
Fr\'{e}d\'{e}ric Haglund and Daniel~T. Wise.
\newblock Special cube complexes.
\newblock {\em Geom. Funct. Anal.}, 17(5):1551--1620, 2008.

\bibitem[HW15a]{HW15}
Mark~F. Hagen and Daniel~T. Wise.
\newblock Cubulating hyperbolic free-by-cyclic groups: the general case.
\newblock {\em Geom. Funct. Anal.}, 25(1):134--179, 2015.

\bibitem[HW15b]{HW}
Tim Hsu and Daniel~T. Wise.
\newblock Cubulating malnormal amalgams.
\newblock {\em Invent. Math.}, 199(2):293--331, 2015.

\bibitem[HW16]{HW16}
Mark~F. Hagen and Daniel~T. Wise.
\newblock Cubulating hyperbolic free-by-cyclic groups: the irreducible case.
\newblock {\em Duke Math. J.}, 165(9):1753--1813, 2016.

\bibitem[KM12]{KM}
Jeremy Kahn and Vladimir Markovic.
\newblock Immersing almost geodesic surfaces in a closed hyperbolic three
  manifold.
\newblock {\em Ann. Math. (2)}, 175(3):1127--1190, 2012.

\bibitem[KPS73]{KPS}
Abraham Karrass, Alfred Pietrowski, and Donald Solitar.
\newblock Finite and infinite cyclic extensions of free groups.
\newblock {\em J. Austral. Math. Soc.}, 16:458--466, 1973.
\newblock Collection of articles dedicated to the memory of Hanna Neumann, IV.

\bibitem[Mar07]{Martino}
Armando Martino.
\newblock A proof that all {S}eifert 3-manifold groups and all virtual surface
  groups are conjugacy separable.
\newblock {\em J. Algebra}, 313(2):773--781, 2007.

\bibitem[Mit98]{MahanCTmaps}
Mahan Mitra.
\newblock Cannon-{T}hurston maps for hyperbolic group extensions.
\newblock {\em Topology}, 37(3):527--538, 1998.

\bibitem[MZ16]{MinasyanZalesskii}
Ashot Minasyan and Pavel Zalesskii.
\newblock Virtually compact special hyperbolic groups are conjugacy separable.
\newblock {\em Comment. Math. Helv.}, 91(4):609--627, 2016.

\bibitem[Nie17]{Nielsen}
Jakob Nielsen.
\newblock Die {I}somorphismen der allgemeinen, unendlichen {G}ruppe mit zwei
  {E}rzeugenden.
\newblock {\em Math. Ann.}, 78(1):385--397, 1917.

\bibitem[Pet97]{pettet_finite_out}
Martin~R. Pettet.
\newblock Virtually free groups with finitely many outer automorphisms.
\newblock {\em Trans. Amer. Math. Soc.}, 349(11):4565--4587, 1997.

\bibitem[Sel97]{SelaStructure}
Zlil Sela.
\newblock Structure and rigidity in ({G}romov) hyperbolic groups and discrete
  groups in rank {$1$} {L}ie groups. {II}.
\newblock {\em Geom. Funct. Anal.}, 7(3):561--593, 1997.

\bibitem[Sta71]{Stallings}
John Stallings.
\newblock {\em Group theory and three-dimensional manifolds}, volume~4 of {\em
  Yale Mathematical Monographs}.
\newblock Yale University Press, New Haven, Conn.-London, 1971.
\newblock A James K. Whittemore Lecture in Mathematics given at Yale
  University, 1969.

\bibitem[Thu82]{Thurston3dmflds}
William~P. Thurston.
\newblock Three-dimensional manifolds, {K}leinian groups and hyperbolic
  geometry.
\newblock {\em Bull. Amer. Math. Soc. (N.S.)}, 6(3):357--381, 1982.

\bibitem[Wis21]{Wise_book}
Daniel~T. Wise.
\newblock {\em The structure of groups with a quasiconvex hierarchy}, volume
  209 of {\em Ann. Math. Stud.}
\newblock Princeton, NJ: Princeton University Press, 2021.

\end{thebibliography}
}
\sc{F. D. Institut Fourier, Univ. Grenoble Alpes, Grenoble, France.\\   IRL-CRM  CNRS, Universit\'e de Montr\'eal, Montréal, Canada.}

\sc{S. K. M. S. Faculty of Mathematics, Technion -- Israel Institute of Technology, Haifa, Israel}

\sc{J. P. M. Department of Mathematics, Princeton University, Princeton, NJ, USA}

\end{document}